\documentclass{ndjflart}
\usepackage[T1]{fontenc}
\usepackage{tgtermes}
\usepackage{mathptmx}
\usepackage[scaled=.92]{helvet}
\usepackage{amsthm,amsmath,amssymb}
\usepackage{mathrsfs}
\usepackage[numbers]{natbib}
\usepackage[colorlinks,citecolor=blue,urlcolor=blue]{hyperref}
\usepackage{enumerate}
\usepackage{bm}

\artstatus{am}

\newtheorem{theorem}{Theorem}[section]
\newtheorem{lemma}[theorem]{Lemma}

\newtheorem{proposition}[theorem]{Proposition}
\newtheorem{example}[theorem]{Example}

\newtheorem{corollary}[theorem]{Corollary}
\theoremstyle{definition}
\newtheorem{definition}[theorem]{Definition}

\theoremstyle{remark}

\date{September 3, 2013}

\def\N{\mathbb{N}}

\def\rest{\upharpoonright}
\def\Rm{\mathrm{Rm}}
\def\mymod{\mbox{ mod }}
\newcommand{\claim}[1]{\vspace{1 mm}\noindent \textbf{Claim #1} \hspace{2 mm}}

\startlocaldefs

\endlocaldefs

\begin{document}

\begin{frontmatter}

\title{\emph{Guessing, Mind-changing, and the Second Ambiguous Class}}

\runtitle{Guessing and Mind-changing}

\author{\fnms{Samuel}
 \snm{Alexander}
 \corref{}
 \ead[label=e1]{alexander@math.ohio-state.edu}
}

\address{Department of Mathematics\\
 The Ohio State University\\
 231 West 18th Ave\\
 Columbus OH 43210
 USA\\
 \printead{e1}
}

\runauthor{S.~Alexander}

\begin{abstract}
In his dissertation, Wadge
defined a notion of guessability on subsets of the Baire space
and gave two characterizations of guessable sets.
A set is guessable iff it is in the
second ambiguous class ($\bm{\Delta}^0_2$),
iff
it is eventually annihilated by
a certain remainder.
We simplify this remainder and
give a new proof of the latter equivalence.
We then introduce a notion of guessing with an
ordinal limit on how often one can change one's mind.
We show that for every ordinal $\alpha$,
a guessable set is annihilated by
$\alpha$ applications of the simplified remainder
if and only if it is guessable with fewer than $\alpha$
mind changes.
We use guessability with fewer than $\alpha$ mind changes
to give a semi-characterization of the Hausdorff difference
hierarchy, and indicate how Wadge's notion of guessability
can be generalized to higher-order guessability, providing
characterizations of $\bm{\Delta}^0_\alpha$ for all
successor ordinals $\alpha>1$.
\end{abstract}



\end{frontmatter}

\section{Introduction}
\label{introsect}

Let $\N^\N$ be the set of sequences $s:\N\to\N$ and let
$\N^{<\N}$ be the set $\cup_n \N^n$ of finite sequences.
If $s\in\N^{<\N}$, we will write $[s]$ for
$\{f\in\N^{\N}\,:\,\mbox{$f$ extends $s$}\}$.
We equip $\N^\N$ with a second-countable topology by
declaring $[s]$ to be a basic open set whenever $s\in\N^{<\N}$.

Throughout the paper, $S$ will denote a subset of $\N^\N$.
We say that $S\in\bm{\Delta}^0_2$ if $S$ is simultaneously
a countable intersection of open sets and a countable union
of closed sets in the above topology.  In classic terminology,
$S\in\bm{\Delta}^0_2$ just in case $S$ is both $G_{\delta}$ and $F_{\sigma}$.

The following notion was discovered by Wadge \cite{wadge} (pp.~141--142) and independently
by this author
\cite{alexander2011}.
\footnote{A third independent usage of the term \emph{guessable},
with similar but not the same meaning, appears in \cite{tsaban} (p.~1280), where a subset 
$Y\subseteq\N^\N$
is called guessable if there is a function $g\in\N^\N$ such that for each $f\in Y$, $g(n)=f(n)$ 
for 
infinitely many $n$.}

\begin{definition}
\label{guessabledefn}
We say $S$ is \emph{guessable} if there is a function
$G:\N^{<\N}\to\{0,1\}$ such that
for every $f\in\N^\N$,
\[
\lim_{n\to\infty} G(f\rest n)
= \chi_S(f)
 = \left\{ \begin{array}{l} \mbox{$1$, if $f\in S$,}\\ \mbox{$0$, if 
$f\not\in S$.}\end{array}\right.
\]
If so, we say $G$ \emph{guesses} $S$, or that $G$ is an 
\emph{$S$-guesser}.
\end{definition}

The intution behind the above notion is captured eloquently by Wadge 
(p.~142, notation changed):
\begin{quote}
\label{wadgequote}
Guessing sets allow us to form an opinion as to whether an element
$f$ of $\N^\N$ is in $S$ or $S^c$, given only a finite initial segment
$f\rest n$ of $f$.
\end{quote}
Game theoretically, one envisions an asymmetric game where $II$ (the guesser) has perfect
information, $I$ (the sequence chooser) has zero information,
and $II$'s winning set consists of all sequences $(a_0,b_0,a_1,b_1,\ldots)$
such that $b_i\to 1$ if $(a_0,a_1,\ldots)\in S$ and $b_i\to 0$ otherwise.

The following result was proved in \cite{wadge} (pp.144--145)
by infinite game-theoretical methods.  The present author found
a second proof
\cite{alexander2011}
using mathematical logical methods.

\begin{theorem}
\label{delta2isguessable}
(Wadge)
$S$ is guessable if and only if $S\in\bm{\Delta}^0_2$.
\end{theorem}

Wadge defined (pp.~113--114) the following
remainder operation.

\begin{definition}
\label{wadgederiv}
For $A,B\subseteq \N^\N$, define $\Rm_0(A,B)=\N^\N$.
For $\mu>0$ an ordinal,
define
\[
\Rm_{\mu}(A,B)
= \bigcap_{\nu<\mu}\left(
\overline{\Rm_{\nu}(A,B)\cap A}
\cap
\overline{\Rm_{\nu}(A,B)\cap B}
\right).
\]
(Here $\overline{\bullet}$ denotes topological closure.)
Write $\Rm_{\mu}(S)$ for $\Rm_{\mu}(S,S^c)$.
\end{definition}

By countability considerations,
there is some (in fact countable) ordinal $\mu$,
depending on $S$, such that $\Rm_{\mu}(S)=\Rm_{\mu'}(S)$
for all $\mu'\geq \mu$;  Wadge writes $\Rm_{\Omega}(S)$
for $\Rm_{\mu}(S)$ for such a $\mu$.  He then proves the following
theorem:

\begin{theorem}
\label{wadgehausdorff}
(Wadge, attributed to Hausdorff)
$S\in\bm{\Delta}^0_2$ if and only if $\Rm_{\Omega}(S)=\emptyset$.
\end{theorem}

In Section \ref{firstsection}, we introduce a simpler remainder
$(S,\alpha)\mapsto S_\alpha$ and use it to give a new proof of Theorem 
\ref{wadgehausdorff}.

In Section \ref{secondsection}, we introduce the notion of $S$ being 
guessable while changing one's mind fewer than $\alpha$ many times ($\alpha\in 
\mathrm{Ord}$) and show that this is equivalent to 
$S_\alpha=\emptyset$.

In Section \ref{diffhiersection}, we show that for $\alpha>0$,
$S$ is guessable while changing one's mind fewer than $\alpha+1$ many times
if and only if at least one of $S$ or $S^c$ is in the $\alpha$th level of the difference hierarchy.

In Section \ref{higherordersection}, we generalize guessability, introducing the notion
of $\mu$th-order guessability ($1\leq\mu<\omega_1$).  We show that $S$ is $\mu$th-order guessable
if and only if $S\in\bm{\Delta}^0_{\mu+1}$.

\section{Guessable Sets and Remainders}
\label{firstsection}

In this section we give a new proof of Theorem \ref{wadgehausdorff}.
We find it easier to work with the following remainder\footnote{In general,
there seems to be a correspondence between remainders on $\N^{\N}$
and remainders on $\N^{<\N}$ that take trees to trees; in
the future we might publish more general work based on this
observation.} which is closely 
related to the remainder defined by Wadge.
For $X\subseteq \N^{<\N}$, we will write $[X]$
to denote the set of infinite sequences all of whose finite initial 
segments
lie in $X$.

\begin{definition}
\label{myderiv}
Let $S\subseteq\N^\N$.
We define $S_{\alpha}\subseteq\N^{<\N}$ ($\alpha\in \mathrm{Ord}$)
by transfinite recursion as follows.
We define 
$S_0=\N^{<\N}$, and
$S_\lambda=\cap_{\beta<\lambda}S_{\beta}$
for every limit ordinal $\lambda$.
Finally,
for every ordinal $\beta$, we define
\[
S_{\beta+1}
=
\{
x\in S_{\beta}\,:\,
\mbox{$\exists x',x''\in [S_{\beta}]$
such that $x\subseteq x'$, $x\subseteq x''$, $x'\in S$, $x''\not\in S$}\}.
\]
We write $\alpha(S)$ for the minimal ordinal $\alpha$ such that 
$S_{\alpha}=S_{\alpha+1}$, and we write $S_{\infty}$ for $S_{\alpha(S)}$.
\end{definition}

Clearly $S_{\alpha}\subseteq S_{\beta}$ whenever $\beta<\alpha$.
This remainder notion is related to Wadge's as follows.

\begin{lemma}
\label{simplifier}
For each ordinal $\alpha$, $\Rm_{\alpha}(S)=[S_{\alpha}]$.
\end{lemma}

\begin{proof}
Since $S_{\alpha}\subseteq S_{\beta}$ whenever $\beta<\alpha$,
for all $\alpha$, we have $S_{\alpha}=\cap_{\beta<\alpha}S_{\beta+1}$
(with the convention that $\cap\emptyset=\N^{<\N}$).
We will show by induction on $\alpha$
that $\Rm_{\alpha}(S)=[S_\alpha]=[\cap_{\beta<\alpha}S_{\beta+1}]$.

Suppose $f\in [\cap_{\beta<\alpha}S_{\beta+1}]$.
Let $\beta<\alpha$.
Let $\mathcal{U}$ be an open set around $f$,
we can assume $\mathcal{U}$ is basic open, so
$\mathcal{U}=[f_0]$, $f_0$ a finite initial segment of $f$.
Since $f\in [\cap_{\beta<\alpha}S_{\beta+1}]$,
$f_0\in S_{\beta+1}$.
Thus there are $x',x''\in [S_\beta]$ extending $f_0$ (hence
in $\mathcal{U}$),
$x'\in S$, $x''\not\in S$.
In other words, $x'\in [\cap_{\gamma<\beta} S_{\gamma+1}]\cap S$
and $x''\in [\cap_{\gamma<\beta} S_{\gamma+1}]\cap S^c$.
By induction, $x'\in \Rm_\beta(S)\cap S$
and $x''\in \Rm_\beta(S)\cap S^c$.  By arbitrariness
of $\mathcal{U}$,
$f\in \overline{\Rm_{\beta}(S)\cap S}\cap \overline{\Rm_{\beta}(S)\cap 
S^c}$.
By arbitrariness of $\beta$,
$f\in \Rm_{\alpha}(S)$.

The reverse inclusion is similar.
\end{proof}

Note that Lemma \ref{simplifier} does not say that $\Rm_{\alpha}(S)=\emptyset$
if and only if $S_\alpha=\emptyset$.  It is (at least a priori) possible that
$S_\alpha\not=\emptyset$ while $[S_\alpha]=\emptyset$.
Lemma \ref{simplifier} does however imply that $\Rm_{\Omega}(S)=\emptyset$
if and only if $S_\infty=\emptyset$, since it is easy to see that
if $[S_\alpha]=\emptyset$ then $S_{\alpha+1}=\emptyset$.
Thus in order to prove Theorem \ref{wadgehausdorff}
it suffices to show that $S$ is guessable if and only if 
$S_{\infty}=\emptyset$.
The $\Rightarrow$ direction requires no additional machinery.

\begin{proposition}
\label{proposition4}
If $S$ is guessable then $S_{\infty}=\emptyset$.
\end{proposition}

\begin{proof}
Let $G:\N^{<\N}\to \{0,1\}$ be an $S$-guesser.
Assume (for contradiction)
$S_\infty\not=\emptyset$ and let $\sigma_0\in S_\infty$.
We will build a sequence on whose initial segments $G$
diverges, contrary to Definition \ref{guessabledefn}.
Inductively suppose we have finite sequences
$\sigma_0\subset_{\not=} \cdots \subset_{\not=} \sigma_k$
in $S_\infty$
such that $\forall 0<i\leq k$, $G(\sigma_i)\equiv i\,\mymod 2$.
Since $\sigma_k\in S_\infty=S_{\alpha(S)}=S_{\alpha(S)+1}$,
there are $\sigma',\sigma''\in[S_\infty]$,
extending $\sigma_k$, with $\sigma'\in S$, $\sigma''\not\in S$.
Choose $\sigma\in\{\sigma',\sigma''\}$ with
$\sigma\in S$ iff $k$ is even.  Then $\lim_{n\to\infty}
G(\sigma\rest n)\equiv k+1\mymod 2$.
Let $\sigma_{k+1}\subset \sigma$ properly extend $\sigma_k$
such that $G(\sigma_{k+1})\equiv k+1\mymod 2$.
Note $\sigma_{k+1}\in S_{\infty}$
since $\sigma\in [S_\infty]$.

By induction, there are $\sigma_0\subset_{\not=}\sigma_1
\subset_{\not=}\cdots$ such that for $i>0$,
$G(\sigma_i)\equiv i\mymod 2$.
This contradicts Definition \ref{guessabledefn}
since $\lim_{n\to\infty} G((\cup_i\sigma_i)\rest n)$ ought to
converge.
\end{proof}

The $\Leftarrow$ direction requires a little machinery.

\begin{definition}
\label{betadefn}
If $\sigma\in\N^{<\N}$, $\sigma\not\in S_\infty$,
let $\beta(\sigma)$ be the least
ordinal such that $\sigma\not\in S_{\beta(\sigma)}$.
\end{definition}

Note that whenever $\sigma\not\in S_\infty$,
$\beta(\sigma)$ is a successor ordinal.

\begin{lemma}
\label{lemma6}
Suppose $\sigma\subseteq\tau$ are finite sequences.
If $\tau\in S_\infty$ then $\sigma\in S_\infty$.
And if $\sigma\not\in S_\infty$, then $\beta(\tau)\leq \beta(\sigma)$.
\end{lemma}

\begin{proof}
It is enough to show that $\forall \beta\in\mathrm{Ord}$,
if $\tau\in S_\beta$ then $\sigma\in S_\beta$.
This is by induction on $\beta$, the limit and zero cases being trivial.
Assume $\beta$ is successor.  If $\tau\in S_\beta$,
this means $\tau\in S_{\beta-1}$
and there are $\tau',\tau''\in [S_{\beta-1}]$ extending $\tau$
with $\tau'\in S$, $\tau''\not\in S$.  Since $\tau'$ and $\tau''$
extend $\tau$, and $\tau$ extends $\sigma$, $\tau'$ and $\tau''$ extend 
$\sigma$; and since $\sigma\in S_{\beta-1}$ (by induction), this shows
$\sigma\in S_\beta$.
\end{proof}

\begin{lemma}
\label{lemma7}
Suppose $f:\N\to\N$, $f\not\in [S_\infty]$.
There is some $i$ such that for all $j\geq i$,
$f\rest j\not\in S_\infty$ and $\beta(f\rest j)=\beta(f\rest i)$.
Furthermore, $f\in [S_{\beta(f\rest i)-1}]$.
\end{lemma}

\begin{proof}
The first part follows from Lemma \ref{lemma6} and the well-foundedness of 
$\mathrm{Ord}$.
For the second part we must show $f\rest k\in S_{\beta(f\rest i)-1}$
for every $k$.
If $k\leq i$,
then $f\rest k\in S_{\beta(f\rest i)-1}$ by Lemma \ref{lemma6}.
If $k\geq i$, then $\beta(f\rest k)=\beta(f\rest i)$ and so
$f\rest k\in S_{\beta(f\rest i)-1}$ since it is in
$S_{\beta(f\rest k)-1}$ by definition of $\beta$.
\end{proof}

\begin{definition}
\label{gsubsdefn}
If $S_\infty=\emptyset$
then we define $G_S:\N^{<\N}\to\{0,1\}$
as follows.
Let $\sigma\in\N^{<\N}$.
Since $S_\infty=\emptyset$, $\sigma\not\in S_\infty$,
so $\sigma\in S_{\beta(\sigma)-1}\backslash S_{\beta(\sigma)}$.
Since $\sigma\not\in S_{\beta(\sigma)}$,
this means for every two extensions $x',x''$ of $\sigma$
in $[S_{\beta(\sigma)-1}]$, either $x',x''\in S$ or $x',x''\in S^c$.
So either all extensions of $\sigma$ in $[S_{\beta(\sigma)-1}]$
are in $S$, or all such extensions are in $S^c$.
\begin{enumerate}
\item[(i)] If there are no extensions of $\sigma$ in 
$[S_{\beta(\sigma)-1}]$,
and $\mathrm{length}(\sigma)>0$,
then let $G_S(\sigma)=G_S(\sigma^-)$
where $\sigma^-$ is obtained from $\sigma$ by removing the last term.
\item[(ii)]
If there are no extensions of $\sigma$ in $[S_{\beta(\sigma)-1}]$,
and $\mathrm{length}(\sigma)=0$,
let $G_S(\sigma)=0$.
\item[(iii)] If there are extensions of $\sigma$ in $[S_{\beta(\sigma)-1}]$
and they are all in $S$, define $G_S(\sigma)=1$.
\item[(iv)] If there are extensions of $\sigma$ in $[S_{\beta(\sigma)-1}]$
and they are all in $S^c$, define $G_S(\sigma)=0$.
\end{enumerate}
\end{definition}

\begin{proposition}
\label{proposition8}
If $S_\infty=\emptyset$ then $G_S$ guesses $S$.
\end{proposition}

\begin{proof}
Assume $S_\infty=\emptyset$.
%
%
Let $f\in S$.
I will show $G_S(f\rest n)\to 1$ as $n\to\infty$.
Since $f\not\in [S_\infty]$, let $i$ be as in Lemma \ref{lemma7}.
I claim $G_S(f\rest j)=1$ whenever $j\geq i$.
Fix $j\geq i$.  We have $\beta(f\rest j)=\beta(f\rest i)$
by choice of $i$, and $f\in [S_{\beta(f\rest i)-1}]
=[S_{\beta(f\rest j)-1}]$.
%
%
Since $f\rest j$ has one extension (namely $f$ itself) in 
both $[S_{\beta(f\rest j)-1}]$ and $S$,
$G_S(f\rest j)=1$.

Identical reasoning shows that if $f\not\in S$ then 
$\lim_{n\to\infty}G_S(f\rest n)=0$.
\end{proof}

\begin{theorem}
\label{theorem9}
$S\in\bm{\Delta}^0_2$ if and only if $S_\infty=\emptyset$.
That is, Theorem \ref{wadgehausdorff} is true.
\end{theorem}

\begin{proof}
By combining Propositions \ref{proposition4} and
\ref{proposition8} and Theorem \ref{delta2isguessable}.
\end{proof}

\section{Guessing without changing one's Mind too often}
\label{secondsection}

In this section our goal is to tease out additional information
about $\bm{\Delta}^0_2$ from the operation defined in Definition
\ref{myderiv}.

\begin{definition}
\label{mindchangedefn}
For each function $G$ with domain $\N^{<\N}$,
if $G(f\rest (n+1))\not= G(f\rest n)$ ($f\in\N^\N$, $n\in\N$),
we say $G$ \emph{changes its mind on $f\rest (n+1)$}.
Now let $\alpha\in\mathrm{Ord}$.
We say $S$ is \emph{guessable with $<\alpha$ mind changes}
if there is an
$S$-guesser $G$
along with a function $H:\N^{<\N}\to\alpha$ such that
the following hold, where $f\in\N^\N$ and $n\in\N$.
\begin{enumerate}
\item[(i)] $H(f\rest (n+1)) \leq H(f\rest n)$.
\item[(ii)] If $G$ changes its mind on $f\rest (n+1)$, then $H(f\rest 
(n+1))<H(f\rest n)$.
\end{enumerate}
\end{definition}

This notion bears some resemblance to the notion of a set $Z\subseteq\N$
being $f$-c.e.~in \cite{figueira}, or $g$-c.a.~in \cite{nies}.

\begin{theorem}
\label{mindchangetheorem}
For $\alpha\in\mathrm{Ord}$, $S$ is guessable with $<\alpha$ mind changes 
if and only if $S_{\alpha}=\emptyset$.
\end{theorem}

\begin{proof}

\item ($\Rightarrow$)
Assume $S$ is guessable with $<\alpha$ mind changes.
Let $G,H$ be as in Definition \ref{mindchangedefn}.
We claim that for all $\beta\in\mathrm{Ord}$,
if $\sigma\in S_{\beta}$ then
$H(\sigma)\geq\beta$.
This will prove ($\Rightarrow$)
because it implies
that if $S_\alpha\not=\emptyset$ then there is
some $\sigma$ with $H(\sigma)\geq\alpha$,
absurd since $\mathrm{codomain}(H)=\alpha$.

We attack the claim by induction on $\beta$.
The zero and limit cases are trivial.
Assume $\beta=\gamma+1$.
Suppose $\sigma\in S_{\gamma+1}$.
There are $x',x''\in [S_{\gamma}]$ extending
$\sigma$, $x'\in S$, $x''\not\in S$.
Pick $x\in\{x',x''\}$
so that $\chi_S(x)\not=G(\sigma)$
and pick $\sigma^+\in\N^{<\N}$ with $\sigma\subseteq \sigma^+\subseteq x$
such that $G(\sigma^+)=\chi_S(x)$ (some such $\sigma^+$ exists
since $G$ guesses $S$).
Since $x\in [S_{\gamma}]$, $\sigma^+\in S_{\gamma}$.
By induction, $H(\sigma^+)\geq \gamma$.
The fact $G(\sigma^+)\not= G(\sigma)$ implies
$H(\sigma^+)<H(\sigma)$, forcing $H(\sigma)\geq \gamma+1$.

\item
($\Leftarrow$)
Assume $S_\alpha=\emptyset$.
%
%
For all $\sigma\in\N^{<\N}$,
define $H(\sigma)=\beta(\sigma)-1$
(by definition of $\beta(\sigma)$, since $S_\alpha=\emptyset$,
clearly
$H(\sigma)\in \alpha$).
I claim $G_S,H$ witness that $S$ is guessable with $<\alpha$
mind changes.

By Proposition \ref{proposition8}, $G_S$ guesses $S$.
Let $f\in\N^\N$, $n\in\N$.
By Lemma \ref{lemma6},
$H(f\rest (n+1))\leq H(f\rest n)$.
Now suppose $G_S$ changes its mind on $f\rest(n+1)$,
we must show $H(f\rest (n+1))< H(f\rest n)$.
Assume, for sake of contradiction,
that $H(f\rest(n+1))=H(f\rest n)$.
Assume $G_S(f\rest n)=0$, the other case is similar.
By definition of $G_S$, ($*$) for every infinite extension $f'$ of $f\rest n$,
if $f'\in [S_{\beta(f\rest n)-1}]$ then
$f'\in S^c$.
Since $G_S$ changes its mind on $f\rest(n+1)$,
$G_S(f\rest(n+1))=1$.
Thus ($**$) for every infinite extension $f''$ of $f\rest(n+1)$,
if $f''\in [S_{\beta(f\rest(n+1))-1}]$
then $f''\in S$.
And $f\rest(n+1)$ does actually have some such infinite extension $f''$, because if it had none,
that would make $G_S(f\rest(n+1))=G_S(f\rest n)$ by case 1 of the definition of $G_S$ (Definition \ref{gsubsdefn}).
Being an extension of $f\rest(n+1)$, $f''$ also extends $f\rest n$; and
by the assumption that $H(f\rest(n+1))=H(f\rest n)$,
$f''\in [S_{\beta(f\rest n)-1}]$.
By ($*$), $f''\in S^c$, and by ($**$), $f''\in S$.
Absurd.
\end{proof}

It is not hard to show $S$ is a Boolean combination of open sets if and only if $S$ is guessable with $<\omega$ mind changes,
so
Theorem \ref{mindchangetheorem} and Lemma \ref{simplifier} give
a new proof of a special case
of the main theorem (p.~1348) of \cite{miller} (see also \cite{allouche}).

\section{Mind Changing and the Difference Hierarchy}
\label{diffhiersection}

We recall the following definition from \cite{kechris} (p.~175, stated in greater generality---we
specialize it to the Baire space).  In this definition, $\bm{\Sigma}^0_1(\N^\N)$
is the set of open subsets of $\N^\N$, and the \emph{parity} of an ordinal $\eta$ is the equivalence
class modulo $2$ of $n$, where $\eta=\lambda+n$, $\lambda$ a limit ordinal (or $\lambda=0$), $n\in\mathbb N$.

\begin{definition}
\label{hierarchydefn}
Let $(A_{\eta})_{\eta<\theta}$ be an increasing sequence of subsets of $\N^\N$
with $\theta\geq 1$.
Define the set $D_{\theta}((A_{\eta})_{\eta<\theta})\subseteq\N^\N$ by
\begin{quote}
$x\in D_{\theta}((A_{\eta})_{\eta<\theta})$ $\Leftrightarrow$ $\displaystyle x\in \bigcup_{\eta<\theta} A_{\eta}$
\& the least $\eta<\theta$ with $x\in A_{\eta}$ has parity opposite to that of $\theta$.
\end{quote}
Let
\[
D_{\theta}(\bm{\Sigma}^0_1)(\N^\N)=\{D_{\theta}((A_{\eta})_{\eta<\theta})\,:\, A_{\eta}\in \bm{\Sigma}^0_1(\N^\N),\,\eta<\theta\}.
\]
\end{definition}

This hierarchy offers a constructive characterization of $\bm{\Delta}^0_2$: it turns out that
\[\bm{\Delta}^0_2=\cup_{1\leq \theta<\omega_1} D_{\theta}(\bm{\Sigma}^0_1)(\N^\N)\]
(see Theorem 22.27 of \cite{kechris}, p.~176, attributed to Hausdorff and Kuratowski).

For brevity, we will write $D_{\alpha}$ for $D_{\alpha}(\bm{\Sigma}^0_1)(\N^\N)$.

\begin{theorem}
\label{hierarchymainthm}
(Semi-characterization of the difference hierarchy)
Let $\alpha>0$.
The following are equivalent.
\begin{enumerate}
\item[(i)] $S$ is guessable with $<\alpha+1$ mind changes.
\item[(ii)] $S\in D_{\alpha}$ or $S^c\in D_{\alpha}$.
\end{enumerate}
\end{theorem}

We will prove Theorem \ref{hierarchymainthm} by a sequence of smaller results.

\begin{definition}
\label{utilitydefn}
For $\alpha,\beta\in\mathrm{Ord}$, write $\alpha\equiv\beta$ to indicate that $\alpha$
and $\beta$ have the same parity (that is, $2|n-m$, where $\alpha=\lambda+n$
and $\beta=\kappa+m$, $n,m\in\N$, $\lambda$ a limit ordinal or $0$,
$\kappa$ a limit ordinal or $0$).
\end{definition}

\begin{proposition}
\label{dalphaisguessable}
Let $\alpha>0$.
If $S\in D_{\alpha}$, say $S=D_{\alpha}((A_\eta)_{\eta<\alpha})$ ($A_\eta\subseteq\N^\N$ open),
then $S$ is guessable with $<\alpha+1$ mind changes.
\end{proposition}

\begin{proof}
Define $G:\N^{<\N}\to\{0,1\}$ and $H:\N^{<\N}\to \alpha+1$ as follows.
Suppose $\sigma\in\N^{<\N}$.
If there is no $\eta<\alpha$ such that $[\sigma]\subseteq A_{\eta}$,
let $G(\sigma)=0$ and let $H(\sigma)=\alpha$.
If there is an $\eta<\alpha$ (we may take $\eta$ minimal) such that $[\sigma]\subseteq A_\eta$,
then let
\begin{align*}
G(\sigma) = \left\{\begin{array}{lr}
\mbox{$0$,} & \mbox{if $\eta\equiv\alpha$;}\\
\mbox{$1$,} & \mbox{if $\eta\not\equiv\alpha$,}
\end{array}\right.
&{}
&
H(\sigma) = \eta.
\end{align*}
Let $f:\N\to\N$.

\item
\claim{1}
$\lim_{n\to\infty} G(f\rest n)=\chi_S(f)$.

If $f\not\in\cup_{\eta<\alpha}A_{\eta}$, then $f\not\in D_\alpha((A_\eta)_{\eta<\alpha})=S$,
and $G(f\rest n)$ will always be $0$, so $\lim_{n\to\infty} G(f\rest n)=0=\chi_S(f)$.
Assume $f\in \cup_{\eta<\alpha}A_\eta$, and let $\eta<\alpha$ be minimum such that $f\in A_\eta$.
Since $A_\eta$ is open, there is some $n_0$ so large that $\forall n\geq n_0$, $[f\rest n]\subseteq A_\eta$.
For all $n\geq n_0$, by minimality of $\eta$, $[f\rest n]\not\subseteq A_{\eta'}$ for any $\eta'<\eta$,
so $G(f\rest n)=0$ if and only if $\eta\equiv\alpha$.
The following are equivalent.
\begin{align*}
f\in S &\mbox{ iff } f\in D_\alpha((A_\eta)_{\eta<\alpha})\\
&\mbox{ iff } \eta\not\equiv\alpha\\
&\mbox{ iff } G(f\rest n)\not=0\\
&\mbox{ iff } G(f\rest n)=1.
\end{align*}
This shows $\lim_{n\to\infty} G(f\rest n)=\chi_S(f)$.

\item
\claim{2}
$\forall n\in\N$, $H(f\rest (n+1))\leq H(f\rest n)$.

If $H(f\rest n)=\alpha$, there is nothing to prove.
If $H(f\rest n)<\alpha$, then $H(f\rest n)=\eta$ where $\eta$
is minimal such that $[f\rest n]\subseteq A_\eta$.
Since $[f\rest (n+1)]\subseteq [f\rest n]$, we have
$[f\rest (n+1)]\subseteq A_\eta$, implying $H(f\rest (n+1))\leq \eta$.

\item
\claim{3}
$\forall n\in\N$, if $G(f\rest (n+1))\not=G(f\rest n)$,
then $H(f\rest (n+1))<H(f\rest n)$.

Assume (for sake of contradiction) $H(f\rest (n+1))\geq H(f\rest n)$.
By Claim 2, $H(f\rest (n+1))=H(f\rest n)$.
By definition of $H$ this implies that $\forall \eta<\alpha$, $[f\rest (n+1)]\subseteq A_\eta$
if and only if $[f\rest n]\subseteq A_\eta$.
This implies $G(f\rest (n+1))=G(f\rest n)$, contradiction.

\item
By Claims 1--3, $G$ and $H$ witness that $S$ is guessable with $<\alpha+1$ mind changes.
\end{proof}

\begin{corollary}
\label{dalphaisguessablecorollary}
Let $\alpha>0$.  If $S\in D_{\alpha}$ or $S^c\in D_\alpha$ then $S$ is guessable with $<\alpha+1$ mind changes.
\end{corollary}

\begin{proof}
If $S\in D_\alpha$ this is immediate by Proposition \ref{dalphaisguessable}.
If $S^c\in D_\alpha$ then Proposition \ref{dalphaisguessable} says $S^c$ is guessable with $<\alpha+1$
mind changes, and this clearly implies that $S$ is too.
\end{proof}

\begin{lemma}
\label{Htechlemma}
Suppose $S$ is guessable with $<\alpha$ mind changes.
Let $G:\N^{<\N}\to\{0,1\}$, $H:\N^{<\N}\to\alpha$ be a pair of functions witnessing as much (Definition \ref{mindchangedefn}).
There is an $H':\N^{<\N}\to\alpha$ such that
$G,H'$ also witness that $S$ is guessable with $<\alpha$ mind changes,
with $H'(\emptyset)=H(\emptyset)$,
and
with the additional property
that for every $f:\N\to\N$ and every $n\in\N$,
\begin{quote}
$H(f\rest (n+1))\equiv H(f\rest n)$ if and only if $G(f\rest (n+1))=G(f\rest n)$.
\end{quote}
\end{lemma}

\begin{proof}
Define $H'(\sigma)$ by induction on the length of $\sigma$ as follows.
Let $H'(\emptyset)=H(\emptyset)$.
If $\sigma\not=\emptyset$,
write $\sigma=\sigma_0\frown n$ for some $n\in\N$ ($\frown$ denotes concatenation).
If $G(\sigma)=G(\sigma_0)$, let $H'(\sigma)=H'(\sigma_0)$.
Otherwise, let $H'(\sigma)$ be either $H(\sigma)$ or $H(\sigma)+1$, whichever has
parity opposite to $H'(\sigma_0)$.

By construction $H'$ has the desired parity properties.
A simple inductive argument shows that ($*$) $\forall\sigma\in\N^{<\N}$,
$H(\sigma)\leq H'(\sigma)<\alpha$.
I claim that for all $f:\N\to\N$ and $n\in\N$, $H'(f\rest (n+1))\leq H'(f\rest n)$,
and if $G(f\rest(n+1))\not=G(f\rest n)$ then $H'(f\rest (n+1))<H'(f\rest n)$.

If $G(f\rest(n+1))=G(f\rest n)$, then by definition $H'(f\rest(n+1))=H'(f\rest n)$
and the claim is trivial.  Now assume $G(f\rest(n+1))\not=G(f\rest n)$.
If $H'(f\rest(n+1))=H(f\rest(n+1))$ then $H'(f\rest(n+1))<H(f\rest n)\leq H'(f\rest n)$
and we are done.  Assume \[H'(f\rest(n+1))\not= H(f\rest(n+1)),\]
which forces that $\mbox{($**$) $H'(f\rest(n+1))=H(f\rest(n+1))+1$}$.
To see that \[H'(f\rest(n+1))<H'(f\rest n),\] assume not ($***$).
By Definition \ref{mindchangedefn},
$H(f\rest (n+1))<H(f\rest n)$, so
\begin{align*}
H(f\rest n) &\geq H(f\rest (n+1))+1 &\mbox{(Basic arithmetic)}\\
&= H'(f\rest(n+1)) &\mbox{(By ($**$))}\\
&\geq H'(f\rest n) &\mbox{(By ($***$))}\\
&\geq H(f\rest n). &\mbox{(By ($*$))}
\end{align*}
Equality holds throughout, and $H'(f\rest (n+1))=H'(f\rest n)$.
Contradiction: we chose $H'(f\rest(n+1))$
with parity opposite to $H'(f\rest n)$.
\end{proof}

\begin{definition}
\label{anticongruentdefn}
For all $G,H$ as in Definition \ref{mindchangedefn},
$f\in\N^\N$, write $G(f)$ for $\lim_{n\to\infty}G(f\rest n)$ (so
$G(f)=\chi_S(f)$) and write $H(f)$ for $\lim_{n\to\infty}H(f\rest n)$.
Write $G\equiv H$ to indicate that $\forall f\in\N^\N$, $G(f)\equiv H(f)$;
write $G\not\equiv H$ to indicate that $\forall f\in\N^\N$, $G(f)\not\equiv H(f)$
(we pronounce $G\not\equiv H$ as ``$G$ is anticongruent to $H$'').
\end{definition}

\begin{lemma}
\label{coheringguessers}
Suppose $G:\N^{<\N}\to\{0,1\}$ and $H:\N^{<\N}\to\alpha$
witness that $S$ is guessable with $<\alpha$ mind changes.
There is an $H':\N^{<\N}\to\alpha$
such that $G,H'$ witness that $S$ is guessable with $<\alpha$ mind changes,
and such that the following hold.
\begin{align*}
\mbox{ If $G(\emptyset)\equiv \alpha$ then $H'\not\equiv G$.}
&{}
&
\mbox{ If $G(\emptyset)\not\equiv \alpha$ then $H'\equiv G$.}
\end{align*}
\end{lemma}

\begin{proof}
I claim that without loss of generality, we may assume
the following ($*$):
\begin{align*}
\mbox{ If $G(\emptyset)\equiv \alpha$ then $H(\emptyset)\not\equiv G(\emptyset)$.}
&{}
&
\mbox{ If $G(\emptyset)\not\equiv \alpha$ then $H(\emptyset)\equiv G(\emptyset)$.}
\end{align*}
To see this,
suppose not: either $G(\emptyset)\equiv\alpha$ and $H(\emptyset)\equiv G(\emptyset)$,
or else $G(\emptyset)\not\equiv\alpha$ and $H(\emptyset)\not\equiv G(\emptyset)$.
In either case, $H(\emptyset)\equiv\alpha$.
If $H(\emptyset)\equiv\alpha$ then $H(\emptyset)+1\not=\alpha$,
and so, since $H(\emptyset)<\alpha$, $H(\emptyset)+1<\alpha$,
meaning we may add $1$ to $H(\emptyset)$ to enforce the assumption.

Having assumed ($*$), we may use Lemma \ref{Htechlemma}
to construct $H':\N^{<\N}\to\alpha$
such that $G,H'$ witness that $S$ is guessable with $<\alpha$ mind changes,
$H'(\emptyset)=H(\emptyset)$, and $H'$ changes parity precisely when $G$ changes
parity.  The latter facts, combined with ($*$), prove the lemma.
\end{proof}

\begin{proposition}
\label{workhorse}
Suppose $G:\N^{<\N}\to\{0,1\}$ and $H:\N^{<\N}\to\alpha+1$ witness that
$S$ is guessable with $<\alpha+1$ mind changes.
If $G(\emptyset)=0$ then $S\in D_\alpha$.
\end{proposition}

\begin{proof}
By Lemma \ref{coheringguessers}
we may safely assume the following:
\begin{align*}
\mbox{ If $G(\emptyset)\equiv \alpha+1$ then $H\not\equiv G$.}
&{}
&
\mbox{ If $G(\emptyset)\not\equiv \alpha+1$ then $H\equiv G$.}
\end{align*}
In other words,
\begin{align*}
(*) \mbox{ If $G(\emptyset)\equiv \alpha$ then $H\equiv G$.}
&{}
&
(**) \mbox{ If $G(\emptyset)\not\equiv \alpha$ then $H\not\equiv G$.}
\end{align*}

For each $\eta<\alpha$, let
\begin{align*}
A_\eta &= \{f\in\N^\N\,:\,H(f)\leq\eta\}.
&\mbox{($H(f)$ as in Definition \ref{anticongruentdefn})}
\end{align*}
I claim $S=D_\alpha((A_\eta)_{\eta<\alpha})$, which will prove the proposition
since each $A_\eta$ is clearly open.

Suppose $f\in S$, I will show $f\in D_\alpha((A_\eta)_{\eta<\alpha})$.
Since $f\in S$, $H(f)\not=\alpha$, because if $H(f)$ were $=\alpha$,
this would imply that $G$ never changes its mind on $f$,
forcing $\lim_{n\to\infty} G(f\rest n)=\lim_{n\to\infty} G(\emptyset)=0$,
contradicting the fact that $G$ guesses $S$.

Since $H(f)\not=\alpha$, $H(f)<\alpha$.
It follows that for $\eta=H(f)$ we have $f\in A_\eta$ and $\eta$ is minimal with this property.

\item
Case 1: $G(\emptyset)\equiv \alpha$.  By ($*$), $H\equiv G$.
Since $f\in S$, $\lim_{n\to\infty}G(f\rest n)=1$,
so $\eta=\lim_{n\to\infty}H(f\rest n)\equiv 1$.  Since $\alpha\equiv G(\emptyset)=0$,
this shows $\eta\not\equiv \alpha$, putting $f\in D_\alpha((A_\eta)_{\eta<\alpha})$.

\item
Case 2: $G(\emptyset)\not\equiv\alpha$.  By ($**$), $H\not\equiv G$.
Since $f\in S$, $\lim_{n\to\infty}G(f\rest n)=1$,
so $\eta=\lim_{n\to\infty}H(f\rest n)\equiv 0$.
Since $\alpha\not\equiv G(\emptyset)=0$, this shows $\eta\not\equiv\alpha$,
so $f\in D_\alpha((A_\eta)_{\eta<\alpha})$.

\item
Conversely, suppose $f\in D_\alpha((A_\eta)_{\eta<\alpha})$, I will show $f\in S$.
Let $\eta$ be minimal such that $f\in A_\eta$
(by definition of $A_\eta$, $\eta=H(f)$).  By definition of $D_\alpha((A_\eta)_{\eta<\alpha})$,
$\eta\not\equiv\alpha$.

\item
Case 1: $G(\emptyset)\equiv\alpha$.  By ($*$), $H\equiv G$.
Since $\lim_{n\to\infty}H(f\rest n)=H(f)=\eta\not\equiv \alpha\equiv G(\emptyset)=0$,
we see $\lim_{n\to\infty}H(f\rest n)=1$.  Since $H\equiv G$,
$\lim_{n\to\infty}G(f\rest n)=1$, forcing $f\in S$ since $G$ guesses $S$.

\item
Case 2: $G(\emptyset)\not\equiv\alpha$.  By ($**$), $H\not\equiv G$.
Since
\[\lim_{n\to\infty}H(f\rest n)=H(f)=\eta\not\equiv\alpha\not\equiv G(\emptyset)=0,\]
we see $\lim_{n\to\infty}H(f\rest n)=0$.
Since $H\not\equiv G$, $\lim_{n\to\infty}G(f\rest n)=1$, again showing $f\in S$.
\end{proof}

\begin{corollary}
\label{rightarrowcorollary}
If $S$ is guessable with $<\alpha+1$ mind changes, then $S\in D_{\alpha}$
or $S^c\in D_{\alpha}$.
\end{corollary}

\begin{proof}
Let $G,H$ witness that $S$ is guessable with $<\alpha+1$ mind changes.
If $G(\emptyset)=0$ then $S\in D_{\alpha}$ by Proposition \ref{workhorse}.
If not, then $(1-G),H$ witness that $S^c$ is guessable with $<\alpha+1$
mind changes, and $(1-G)(\emptyset)=0$, so $S^c\in D_{\alpha}$ by Proposition \ref{workhorse}.
\end{proof}

Combining Corollaries \ref{dalphaisguessablecorollary} and \ref{rightarrowcorollary}
proves Theorem \ref{hierarchymainthm}.

%
%

\section{Higher-order Guessability}
\label{higherordersection}

In this section we introduce a notion that generalizes guessability to provide
a characterization for $\bm{\Delta}^0_{\mu+1}$
($1\leq\mu<\omega_1$).  We will show that $S\in \bm{\Delta}^0_{\mu+1}$
if and only if $S$ is $\mu$th-order guessable.
Throughout this section, $\mu$ denotes an ordinal in $[1,\omega_1)$.

\begin{definition}
\label{guessablebaseddefn}
Let $\mathscr{S}=(S_0,S_1,\ldots)$ be a countably infinite
tuple of subsets $S_i\subseteq\N^\N$.
\begin{enumerate}
\item[(i)]
For every $f\in \N^\N$,
write $\mathscr{S}(f)$ for the sequence
$(\chi_{S_0}(f),\chi_{S_1}(f),\ldots)\in\{0,1\}^\N$.
\item[(ii)]
We say that $S$ is \emph{guessable based on $\mathscr{S}$}
if there is a function \[G:\{0,1\}^{<\N}\to\{0,1\}\]
(called an \emph{$S$-guesser based on $\mathscr{S}$}) such that
$\forall f\in\N^\N$,
\[
\lim_{n\to\infty} G(\mathscr{S}(f)\rest n) = \chi_S(f).
\]
\end{enumerate}
\end{definition}

Game theoretically, we envision a game where $I$ (the sequence chooser) has zero information
and $II$ (the guesser) has possibly \emph{better-than-perfect} information:
$II$ is allowed to ask (once per turn) whether $I$'s sequence lies in various $S_i$.
For each $S_i$, player $I$'s act
(by answering the question) of committing to play a sequence in $S_i$ or in $S_i^c$
is similar to the act (described in \cite{martin}, p.~366) of choosing a $I$-imposed subgame.

\begin{example}
\label{higherorderexample}
If $\mathscr{S}$ enumerates the sets of the form $\{f\in\N^\N\,:\,f(i)=j\}$, $i,j\in\N$
then it is not hard to show that $S$ is guessable (in the sense
of Definition \ref{guessabledefn}) if and only if $S$ is guessable based on $\mathscr{S}$.
\end{example}

\begin{definition}
\label{muthorderguessabledefn}
We say $S$ is \emph{$\mu$th-order guessable} if there
is some $\mathscr S=(S_0,S_1,\ldots)$ as in Definition \ref{guessablebaseddefn}
such that the following hold.
\begin{enumerate}
\item[(i)] $S$ is guessable based on $\mathscr S$.
\item[(ii)] $\forall i$, $S_i\in\bm{\Delta}^0_{\mu_i+1}$ for some $\mu_i<\mu$.
\end{enumerate}
\end{definition}

\begin{theorem}
\label{higherorderguessabilitycharacterization}
$S$ is $\mu$th-order guessable if and only if $S\in\bm{\Delta}^0_{\mu+1}$.
\end{theorem}

In order to prove Theorem \ref{higherorderguessabilitycharacterization} we will assume
the following result, which is a specialization and rephrasing of Exercise 22.17 of \cite{kechris} (pp.~172--173,
attributed to Kuratowski).

\begin{lemma}
\label{kuratowskilemma}
The following are equivalent.
\begin{enumerate}
\item[(i)] $S\in\bm{\Delta}^0_{\mu+1}$.
\item[(ii)] There is a sequence $(A_i)_{i\in\N}$, each 
$A_i\in\bm{\Delta}^0_{\mu_i+1}$ for some $\mu_i<\mu$,
such that
\[
S=\bigcup_n\bigcap_{m\geq n}A_m=\bigcap_n\bigcup_{m\geq n}A_m.
\]
\end{enumerate}
\end{lemma}

\begin{proof}[Proof of Theorem \ref{higherorderguessabilitycharacterization}]
\item
($\Rightarrow$)
Let $\mathscr S=(S_0,S_1,\ldots)$ and $G$ witness that $S$ is $\mu$th-order guessable
(so each $S_i\in\bm{\Delta}^0_{\mu_i+1}$ for some $\mu_i<\mu$).
For all $a\in\{0,1\}$ and $X\subseteq\N^\N$, define
\[
X^a = \left\{\begin{array}{lr}\mbox{$X$,} & \mbox{ if $a=1$;}\\ \mbox{$\N^\N\backslash X$,} & \mbox{ if $a=0$.}\end{array}\right.
\]
For notational convenience, we will write ``$G(\vec{a})=1$'' as an abbreviation for
``$0\leq a_0,\ldots,a_{m-1}\leq 1$ and $G(a_0,\ldots,a_{m-1})=1$,'' provided $m$ is clear from context.
Observe that for all $f\in\N^\N$ and $m\in\N$, $G(\mathscr{S}(f)\rest m)=1$ if and only if
\[
f\in \bigcup_{G(\vec{a})=1}\bigcap_{j=0}^{m-1}S^{a_j}_j.
\]
Now, given $f:\N\to\N$, $f\in S$ if and only if $G(\mathscr S(f)\rest n)\to 1$,
which is true if and only if $\exists n\forall m\geq n$, $G(\mathscr S(f)\rest m)=1$.
Thus
\begin{align*}
f\in S &\mbox{ iff } \exists n\forall m\geq n,\,G(\mathscr S(f)\rest m)=1\\
&\mbox{ iff } \exists n\forall m\geq n,\, f\in \bigcup_{G(\vec{a})=1}\bigcap_{j=0}^{m-1} S^{a_j}_j\\
&\mbox{ iff } f\in \bigcup_n \bigcap_{m\geq n}
\bigcup_{G(\vec{a})=1}\bigcap_{j=0}^{m-1} S^{a_j}_j.
\end{align*}
So
\[
S = \bigcup_n \bigcap_{m\geq n} \bigcup_{G(\vec{a})=1}\bigcap_{j=0}^{m-1} S^{a_j}_j.
\]
At the same time, since $G(\mathscr{S}(f)\rest m)\to 0$ whenever $f\not\in S$,
we see $f\in S$ if and only if $\forall n\exists m\geq n$ such that $G(\mathscr S(f)\rest m)=1$.
Thus by similar reasoning to the above,
\[
S = \bigcap_n \bigcup_{m\geq n} \bigcup_{G(\vec{a})=1}\bigcap_{j=0}^{m-1} S^{a_j}_j.
\]
For each $m$, $\bigcup_{G(\vec{a})=1}\bigcap_{j=0}^{m-1} S^{a_j}_j$ is a finite
union of finite intersections of sets in $\bm{\Delta}^0_{\mu'+1}$
for various $\mu'<\mu$, thus $\bigcup_{G(\vec{a})=1}\bigcap_{j=0}^{m-1} S^{a_j}_j$
itself is in $\bm{\Delta}^0_{\mu_m+1}$ for some $\mu_m<\mu$.
Letting $A_m=\bigcup_{G(\vec{a})=1}\bigcap_{j=0}^{m-1} S^{a_j}_j$, Lemma \ref{kuratowskilemma}
says $S\in\bm{\Delta}^0_{\mu+1}$.

\item
($\Leftarrow$)
Assume $S\in\bm{\Delta}^0_{\mu+1}$.
By Lemma \ref{kuratowskilemma}, there are $(A_i)_{i\in\N}$,
each $A_i\in\bm{\Delta}^0_{\mu_i+1}$ for some $\mu_i<\mu$, such that
\begin{align*}
S &= \bigcup_n\bigcap_{m\geq n}A_m = \bigcap_n\bigcup_{m\geq n}A_m. &\mbox{($*$)}
\end{align*}
I claim that $S$ is guessable based on $\mathcal S=(A_0,A_1,\ldots)$.
Define $G:\{0,1\}^{<\N}\to\{0,1\}$ by $G(a_0,\ldots,a_m)=a_m$, I will show that
$G$ is an $S$-guesser based on $\mathcal{S}$.

Suppose $f\in S$.  By ($*$), $\exists n$ s.t.~$\forall m\geq n$, $f\in A_m$ and thus
$\chi_{A_m}(f)=1$.
For all $m\geq n$,
\begin{align*}
G(\mathcal{S}(f)\rest (m+1)) &= G(\chi_{A_0}(f),\ldots,\chi_{A_m}(f))\\
&= \chi_{A_m}(f)\\
&= 1,
\end{align*}
so $\lim_{n\to\infty} G(\mathcal{S}(f)\rest n)=1$.
A similar argument shows that if $f\not\in S$ then $\lim_{n\to\infty} G(\mathcal{S}(f)\rest n)=0$.
\end{proof}

Combining Theorems \ref{delta2isguessable} and \ref{higherorderguessabilitycharacterization},
we see that $S$ is guessable if and only if $S$ is $1$st-order guessable.
It is also not difficult to give a direct proof of this equivalence,
and having done so, Theorem \ref{higherorderguessabilitycharacterization}
provides yet another proof of Theorem \ref{delta2isguessable}.

\begin{acks}
We acknowledge Tim Carlson, Chris Miller, Dasmen Teh, and Erik Walsberg for
many helpful questions and suggestions.  We are gratetful to a referee of an
earlier manuscript for making us aware of William Wadge's dissertation.
\end{acks}

\end{document}